\documentclass[11pt,a4paper, twoside]{amsart}

\usepackage{url}
\usepackage{amssymb}
\usepackage{color}
\usepackage{fancyhdr}
\usepackage{colonequals}
\usepackage{graphicx}
\usepackage{pinlabel}

\usepackage{labelfig}
\usepackage{epsfig}
\usepackage{epstopdf}

\usepackage{pinlabel,cite}
\input{epsf.tex}
\usepackage{latexsym,amsfonts,amssymb,verbatim,mathrsfs,amsthm}
\usepackage{amsmath,amsthm,amssymb,latexsym,graphics,textcomp}
\usepackage{eucal,eufrak}
\usepackage{colonequals}
\usepackage{graphicx}
\usepackage{url}
\input{xy}
\xyoption{all}

\usepackage{enumitem}

\input{epsf.tex}

\renewcommand{\epsilon}{\varepsilon}
\renewcommand{\setminus}{\smallsetminus}

\newtheorem*{namedtheorem}{\theoremname}
\newcommand{\theoremname}{testing}

\newtheorem{theorem}{Theorem}[section]
\newtheorem{proposition}[theorem]{Proposition}
\newtheorem{corollary}[theorem]{Corollary}
\newtheorem{lemma}[theorem]{Lemma}

\newtheorem{fact}[theorem]{Fact}

\theoremstyle{definition}

\theoremstyle{remark}
\newtheorem*{remark}{Remark}

\newcommand{\CA}{\mathcal A}
\newcommand{\CG}{\mathcal C}

\newcommand{\cohom}[3]{H^{{\raise1pt\hbox{$\scriptstyle#1$}}}(#2\>\!,#3)}
\newcommand{\tatecohom}[3]%
  {\widehat H^{{\raise1pt\hbox{$\scriptstyle#1$}}}(#2\>\!,#3)}

\newcommand{\Cohom}[3]%
  {H^{{\raise1pt\hbox{$\scriptstyle#1$}}}\big(#2\>\!,#3\big)}
\newcommand{\Tatecohom}[3]%
  {\widehat H^{{\raise1pt\hbox{$\scriptstyle#1$}}}\big(#2\>\!,#3\big)}

\newcommand{\homol}[3]{H_{{\lower1pt\hbox{$\scriptstyle#1$}}}(#2\>\!,#3)}
\newcommand{\homolog}[2]{H_{{\lower1pt\hbox{$\scriptstyle#1$}}}(#2)}

\begin{document}

\title[]{Arc and curve graphs for infinite-type surfaces}

\author{Javier Aramayona, Ariadna Fossas \& Hugo Parlier}

\date{\today}

\maketitle

\begin{abstract}
We study arc graphs and curve graphs for surfaces of infinite topological type. First, we define an arc graph relative to a finite number of (isolated) punctures and prove that it is a connected, uniformly hyperbolic graph of infinite diameter; this extends a recent result of J. Bavard to a large class of  punctured surfaces. 

Second, we study the subgraph of the curve graph  spanned by those elements which intersect a fixed separating curve on the surface. We show that this graph has infinite diameter and geometric rank 3, and thus is not hyperbolic. 

\end{abstract}

\section{Introduction}

For surfaces of finite topological type, an important number of problems about mapping class groups and Teichm\"uller spaces may be understood in terms of the various complexes constructed from curves and/or arcs. Prominent examples of these are the {\em curve graph} and the {\em arc graph} (see Section \ref{sec:prelim-graphs} for definitions); an important feature of both is that they are  hyperbolic, see \cite{MM1} and \cite{MS} respectively. 

On the other hand, these complexes have received limited attention in the case of surfaces of infinite topological type, mainly due to the fact that mimicking the definitions from the case of finite-type surfaces ends up producing a graph of finite diameter; compare with Section 3. However, J. Bavard \cite{Bavard} has recently proved that a natural subgraph of the arc graph is hyperbolic and has infinite diameter (with respect to its intrinsic metric), in the case when the surface is homeomorphic to $\mathbb{S}^2$ minus the union of the north pole and a Cantor set. The main objective of this paper is to extend Bavard's result to a large class of punctured surfaces of infinite topological type.

\subsection{Arc graphs} Let $\Sigma$ be a connected orientable surface of infinite topological type and with empty boundary. Let $\Pi \subset \Sigma$ be the set of punctures of $\Sigma$, which we will always assume to be non-empty. It will also be useful to regard the elements of $\Pi$ as marked points on $\Sigma$, and we will feel free to switch between the two viewpoints in the sequel. Throughout, we will need to assume that $\Pi$ {\em contains at least one point that is isolated in $\Pi$}, when $\Pi$ is equipped with the subspace topology. 
 
 Given a set $P$ of isolated punctures, we define $\CA(\Sigma, P)$ to be the simplicial graph whose vertices correspond to isotopy classes, relative to endpoints, of arcs on $\Sigma$ with both endpoints in $P$, and where two such arcs are adjacent in $\CA(\Sigma, P)$ if they can be realized disjointly on $\Sigma$; see Section \ref{sec:prelim-graphs} for an expanded definition. The graph  $\CA(\Sigma, P)$  turns into a metric space by deeming each edge to have unit length. We will prove:

\begin{theorem}
Let $\Sigma$ be a connected orientable surface, of infinite topological type and with at least one puncture. For any finite set $P$ of isolated  punctures, the graph $\CA(\Sigma, P)$ is connected, has infinite diameter, and is  $7$-hyperbolic.  
\label{thm:archyp}
\end{theorem}

\begin{remark}
As mentioned above, in the particular case when $\Sigma$ is homeomorphic to $\mathbb{S}^2$ minus the union of the north pole and a Cantor set, Theorem \ref{thm:archyp} is due to J. Bavard  \cite{Bavard}. The definition of the arc graph in this case was previously suggested by D. Calegari in his blog where he suggested to use this graph to study the existence of non-trivial quasimorphisms from the mapping class group of ${\mathbb S}^2 - K$. 
\end{remark}

Theorem \ref{thm:archyp} could be regarded as a natural extension to the case of infinite-type surfaces of a result of Hensel-Przytycki-Webb \cite{HPW}, stated as  Theorem \ref{thm:hpw} below, which asserts that arc graphs of finite-type surfaces are 7-hyperbolic. In fact, our proof relies heavily on this result.

However, it should be pointed out that, in spite of the similarities between the cases of finite- and infinite-type surfaces,  the analogy between the two situations has  limitations: indeed, a recent result of Bavard-Genevois \cite{BG} asserts that, unlike for finite-type surfaces, the mapping class group  of an infinite-type surface is not {\em acylindrically hyperbolic}; see \cite{osin} for a definition. 

\medskip

\noindent{\em A remark on the geometry of the different arc graphs.}
Consider the ``full"  arc graph $\CA(\Sigma)$, whose vertices are {\em arbitrary} arcs with endpoints in $\Pi$, and where adjacency corresponds to disjointness. One easily sees that, as long as $\Pi$ is infinite, the diameter of $\CA(\Sigma)$ is equal to 2. This fact serves as justification for having to consider arc graphs {\em relative} to a finite set $P$ of {\em isolated} punctures (as we will note in Section 3, the fact that the elements of $P$ be isolated is essential to obtaining a graph of infinite diameter).

This said, the geometry of the different arcs graphs depends heavily on the subset of arcs used to define the given graph. On the one hand, the fact that arcs have {\em both} endpoints in $P$ versus having {\em at least one} endpoint in $P$ is unimportant, as the two definitions produce graphs that are quasi-isometric; see Lemma \ref{lem:qigraphs} below. 
However, in sharp contrast J. Bavard \cite{Bavard2} has proved that the subgraph of $\CA(\Sigma, P)$ spanned by those arcs that have {\em exactly one}  endpoint in $P$ is not hyperbolic whenever $S$ has genus $\ge 1$ or $|P| \ge 2$. 

%

\subsection{Curve graphs} Next, we turn our attention to curve graphs  for surfaces of infinite topological type. Recall that, given a connected orientable surface $\Sigma$, the curve graph $\CG(\Sigma)$ is the simplicial graph whose vertices are isotopy classes of essential simple closed curves on $\Sigma$, and where two such curves are adjacent in $\CG(\Sigma)$ if they have disjoint representatives; see Section \ref{sec:prelim-graphs} for an expanded definition. 

As was the case for arc graphs, when $\Sigma$ has infinite topological type the graph $\CG(\Sigma)$ has diameter 2 and thus has limited geometric interest. One natural way of producing a curve graph of infinite diameter is to only consider curves on $\Sigma$ that intersect a fixed ``portion" of $\Sigma$. More concretely, fix a separating curve $\alpha$ on $\Sigma$ and consider the full subgraph $\CG(\Sigma, \alpha)$ of $\CG(\Sigma)$ spanned by those curves that essentially intersect $\alpha$. 
Our second result states that  $\CG(\Sigma, \alpha)$ has infinite diameter but is never hyperbolic: 

\begin{theorem}\label{thm:curve}
Let $\Sigma$ be a connected orientable surface, of complexity $\ge 1$ if $\Sigma$ has finite topological type. If $\alpha\subset \Sigma$ is a separating curve, the graph $\CG(\Sigma, \alpha)$ has geometric rank 3. 
\end{theorem}
Recall that the {\em geometric rank} of a metric space $X$ is the largest integer $n$ for which there is a quasi-isometric embedding of $\mathbb R^n$ into $X$; see Section \ref{sec:prelim-graphs} for an expanded definition. 



We close the introduction by remarking that examples of non-hyperbolic complexes of infinite diameter were previously constructed by Fossas-Parlier \cite{FP}. Among these complexes is a type of pants graph for infinite type surfaces but, unlike the graphs we consider here, the graphs depend on the geometry of a fixed hyperbolic surface. However, they are very far from being hyperbolic as they have infinite geometric rank.

\medskip 
 
The plan of the paper is as follows. In Section \ref{sec:prelim-metric} we will remind a few notions in metric geometry that will be used. Section \ref{sec:prelim-graphs} contains all the necessary background on arc and curve graphs. Section 4 deals with the proof of Theorem \ref{thm:archyp}. Finally, in Section 5 we will give a proof of Theorem \ref{thm:curve}. 

\medskip

\noindent{\bf Acknowledgements.} The first named author is supported by a Ram\'on y Cajal grant RYC-2013-13008. The second author is supported by ERC grant agreement number 267635 - RIGIDITY. The third author was supported by Swiss National Science Foundation grants numbers PP00P2\textunderscore 128557 and PP00P2\textunderscore 153024. The authors acknowledge support from U.S. National Science Foundation grants DMS 1107452, 1107263, 1107367 ÒRNMS: Geometric structures And Representation varietiesÓ (the GEAR Network). We would like to thank Piotr Przytycki  and Richard Webb for helpful conversations, and in particular for Proposition \ref{cor:hpw}. We would also like to thank Julie Bavard for conversations, and in particular for Proposition 3.5 and the first remark of Section 4. Finally, we are grateful to Ken Bromberg for pointing out a mistake in an earlier version of the paper. 

\section{Preliminaries on metric spaces}
\label{sec:prelim-metric}

We briefly recall the notions of Gromov hyperbolicity, quasi-isometry, and geometric rank. A nice discussion on these topics may be found in \cite{GH}.

\subsection{Hyperbolicity} Let $(X,d)$ be a geodesic metric space, and $\delta \ge 0$. A geodesic triangle $T\subset X$ has a {\em $\delta$-center} if there exists $c\in X$ whose distance to each of the three sides of $T$ is at most $\delta$. We say that $(X,d)$ is {\em $\delta$-hyperbolic} if every geodesic triangle in $X$ has a $\delta$-center; we will simply say that $X$ is hyperbolic if it is $\delta$-hyperbolic for some $\delta$. 

\subsection{Quasi-isometries} Let $(X,d_X)$ and $(Y,d_Y)$ be two metric spaces. We say that a map $f:X\to Y$ is a {\em $(\lambda,C)$-quasi-isometric embedding} if there exist $\lambda \ge 1$ and $C \ge 0$ such that  $$\frac{1}{\lambda}\cdot d_X(x,y) - C \le d_Y(f(x),f(y)) \le \lambda \cdot d_X(x,y) + C$$ for all $x,y\in X$. We will say that the map $f$ above is a {\em quasi-isometry} if, in addition, it is {\em almost surjective}: there exists $\delta>0$ such that every element of $Y$ is at distance at most $\delta$ from an element of $f(X)$. 

We state the following observation as a separate lemma, as it is the way in which we will normally check that a given map is a quasi-isometry: 

\begin{lemma}
Let $(X,d_X)$ and $(Y,d_Y)$ be two metric spaces, and $f:X \to Y$ (resp. $g:Y \to X$) an $L$-Lipschitz (resp. $K$-Lipschitz) map. Suppose that there exist $D_f,D_g\ge 0$ such that $d(x,g\circ f (x)) \le D_f$ 
and   $d(y,f\circ g (y)) \le D_g$ for all $x \in X$ and $y\in Y$. Then $f$ (resp. $g$) is a quasi-isometry, with quasi-isometry constants that depend only on $L$ and $D_f$ (resp, $K$ and $D_g$). 
\label{lem:qitech}
\end{lemma}

Another well-known fact that we will heavily use is that {\em hyperbolicity is invariant under quasi-isometries}: if $X$ is $\delta$-hyperbolic and $f:X\to Y$ is a $(\lambda, C)$-quasi-isometry, then $Y$ is $\delta'$ hyperbolic, where $\delta'$ depends (in an explicit way) only on $\delta$, $\lambda$ and $C$. 

\subsection{Geometric rank} As mentioned in the introduction, the {\em geometric rank} of a metric space $X$ is the largest $n\in\mathbb N$ for which there exists a quasi-isometric embedding $\mathbb R^n \to X$. We note that an unbounded hyperbolic space has geometric rank 1; a nice exercise proves that the geometric rank is invariant under quasi-isometries. 

\section{Arc and curve graphs}
\label{sec:prelim-graphs}

In this section we define arc graphs and curve graphs, and state some results that will be used in the sequel. Throughout, $\Sigma$ will be a connected orientable surface, possibly of infinite topological type, and with empty boundary. Recall that, in the case when $\Sigma$ has finite topological type, the {\em complexity} of $\Sigma$ is defined to be the number $3g-3+p$, where $g$ and $p$ are, respectively, the genus and number of punctures of $\Sigma$.

\subsection{Arc graphs}
 Assume that  $\Sigma$ has at least one puncture, and let $\Pi \subset \Sigma$ be the set of punctures of $\Sigma$; as mentioned in the introduction, we will often regard the elements of $\Pi$ as marked points on $\Sigma$.  By an arc on $\Sigma$ we mean a nontrivial isotopy class (rel endpoints) of arcs on $\Sigma$, with endpoints in $\Pi$, and whose interior is disjoint from $\Pi$.

The {\em arc graph} $\CA(\Sigma)$ is the simplicial graph whose vertices correspond to arcs on $\Sigma$, and where two arcs are adjacent in $\CA(\Sigma)$ if they have representatives with disjoint interiors. The arc graph becomes a geodesic metric space by declaring the length of each edge to be 1. Masur-Schleimer proved that $\CA(\Sigma)$ is $\delta$-hyperbolic \cite{MS}; improving on this result,  Hensel-Przytycki-Webb \cite{HPW} recently proved that arc graphs of finite-type surfaces are {\em uniformly} hyperbolic:

\begin{theorem}[Hensel-Przytycki-Webb \cite{HPW}]
Let $\Sigma$ be a surface of finite topological type, with at least one puncture. The arc graph $\CA(\Sigma)$ is 7-hyperbolic. 
\label{thm:hpw}
\end{theorem}

On the other hand, as we mentioned in the introduction the arc graph has little interest from a geometric point of view as long as $\Sigma$ has infinitely many punctures, due to the following immediate observation: 

\begin{fact}
Let $\Sigma$ be a connected orientable surface with an infinite number of punctures. Then $\CA(\Sigma)$ has diameter 2. 
\label{lem:arcdiam2}
\end{fact}

%
%
%
%

In order to overcome this obstacle, we will consider arc graphs {\em relative} to a finite set of punctures on $\Sigma$. More concretely,  let $P\subset \Pi$ be a finite set, and consider the subgraph $\CA(\Sigma, P)\subset \CA(\Sigma)$ spanned by arcs on $\Sigma$ with both endpoints in $P$. A minor adaptation of the arguments in \cite{HPW} yields that, for $\Sigma$ of finite topological type,  these relative arc graphs are uniformly hyperbolic also, which will constitute a central ingredient of the proof of Theorem \ref{thm:archyp}. 

\begin{proposition}[Hensel-Przytycki-Webb \cite{HPW}]
Let $\Sigma$ be a surface of finite topological type, and $P$ a non-empty set of punctures of $\Sigma$. Then $\CA(\Sigma,P)$ is 7-hyperbolic. 
\label{cor:hpw}
\end{proposition}

Thus we see that our Theorem \ref{thm:archyp} is a natural extension of the above result to the context of infinite-type surfaces, provided these satisfy certain topological conditions.

Before we continue, recall from the introduction that choosing arcs to have both endpoints in $P$ as opposed to having {\em at least} one endpoint in $P$ is not important from the point of view of their large-scale geometry. More concretely, let $P$ be a finite set of punctures on $\Sigma$ such that every element is {\em isolated} in $\Pi$, equipped with the subspace topology. Consider the subgraph  $\CA^*(\Sigma, P)$ of $\CA(\Sigma)$ spanned by those arcs with at least one endpoint in $P$. We have: 

\begin{lemma}\label{lem:Astar}
Let $\Sigma$ be a connected orientable surface, of complexity at least 2 if it has finite topological type. The graphs $\CA(\Sigma, P)$ and $\CA^*(\Sigma, P)$ are quasi-isometric, with quasi-isometry constants that do not depend on $\Sigma$ or $P$. In particular, for any finite-type surface $\Sigma$, $\CA^*(\Sigma, P)$ is $\delta$-hyperbolic for a universal constant $\delta$ that does not depend on $\Sigma$ or $P$. 
\label{lem:qigraphs}
\end{lemma}

\begin{proof}
We suppose that $\Sigma$ has infinite topological type, as the finite-type case is easier. 
First, the natural inclusion map $$\iota: \CA(\Sigma, P)\to  \CA^*(\Sigma, P)$$ is 1-Lipschitz.  In the other direction, there is a map $$\phi: \CA^*(\Sigma, P)\to \CA(\Sigma, P)$$ defined as follows: given an arc $a\in \CA^*(\Sigma, P)$, if $a\in \CA(\Sigma, P)$ we define $\phi(a) =a$. Otherwise, suppose $a$ has endpoints $p\in P$ and $q\notin P$. By the classification of infinite-type surfaces, the set of punctures of $\Sigma$ is a subset of a Cantor set (see Proposition 5 of \cite{richards}). In the light of this, there exists a small regular neighbourhood of $a$ whose boundary does not contain any punctures. We define $\phi(a)$ to be the boundary of any such regular neighborhood, homotoped so that it is based at $p$ (see Figure \ref{f:map}); here we are making use of the fact that $p$ is isolated. We note that, while there may be many choices for $\phi(a)$, any two choices of $\phi(a)$ are at distance at most 2 in $\CA(\Sigma, P)$; to see this, observe that at least one connected component of the complement of the union of two such regular neighborhoods contains an element of $\Pi$. 

For the same reason, the images under the map $\phi$ of two disjoint arcs in $\CA^*(\Sigma, P)$ are at distance at most 2 in $\CA(\Sigma, P)$; in other words, $\phi$ is 2-Lipschitz. 

Finally, observe that $a= \phi \circ \iota(a)$ for all $a \in \CA(\Sigma, P)$. Similarly, $$d(b, \iota\circ \phi (b))\le 1$$ for all $b \in \CA^*(\Sigma, P)$ The result now follows from Lemma \ref{lem:qitech}. 
\end{proof}

%
%


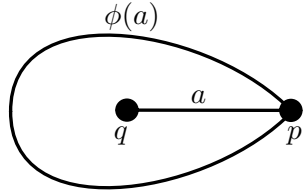
\begin{figure}[h]
\leavevmode \SetLabels
\L(.44*1.02) $\phi(a)$\\
\L(.53*.55) $a$\\
\L(.45*.34) $q$\\
\L(.63*.34) $p$\\
\endSetLabels
\begin{center}
\AffixLabels{\centerline{\epsfig{file =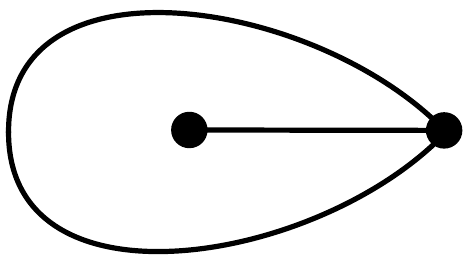,width=4.0cm,angle=0} }}
\vspace{-10pt}
\end{center}
\caption{The image of the arc $a$ under the map $\phi$} \label{f:map}
\end{figure}

\begin{remark}
Note that, a fortiori, the second assertion of Lemma \ref{lem:qigraphs} also holds for infinite-type surfaces, in the light of Theorem \ref{thm:archyp}.
\end{remark}

In sharp contrast, the words ``at least" in the definition of $\CA^*(\Sigma, P)$ are crucial, in the light of the following result due to J. Bavard \cite{Bavard2}. 

\begin{proposition}[\cite{Bavard2}]
Let $\Sigma$ be a surface of infinite topological type, and $P$ a finite set of isolated punctures. Consider the subgraph $\CA_*(\Sigma,P)$  of $\CA(\Sigma, P)$ spanned by those arcs with exactly one endpoint in $P$. If $S$ has genus at least 1, or if $|P| \ge 2$, then $\CA_*(\Sigma,P)$ is not hyperbolic.  
\end{proposition}

\subsection{Curve graphs} We say that a simple closed curve on $\Sigma$ is {\em essential} if it does not bound a disk with at most one puncture. By a {\em curve} on $\Sigma$ we will mean the free isotopy class of an esential simple closed curve on $\Sigma$. 

The {\em curve graph} $\CG(\Sigma)$ is the simplicial graph whose vertices are curves on $\Sigma$, and where two curves are adjacent in $\CG(\Sigma)$ if they can be realized disjointly on $\Sigma$. As before, the curve graph turns into a geodesic metric space by deeming each edge to have length 1. A celebrated theorem of Masur-Minsky asserts that $\CG(\Sigma)$ is hyperbolic whenever $\Sigma$ has finite type: 

\begin{theorem}[Masur-Minsky\cite{MM1}]
Let $\Sigma$ be a connected orientable surface of finite topological type. If $\CG(\Sigma)$ is connected then it is hyperbolic. 
\end{theorem} 

As was the case with the arc graph, in the case of surfaces of infinite topological type the curve graph is not that interesting from a geometric viewpoint:

\begin{fact}
Suppose $\Sigma$ is a connected orientable surface of infinite topological type. Then $\CG(\Sigma)$ has diameter 2. 
\label{lem:diam2}
\end{fact}
%

As mentioned in the introduction, one way of producing a ``curve graph" of infinite diameter is to only consider curves on $\Sigma$ that intersect a fixed curve on the surface. More concretely, fix a separating curve $\alpha \subset \Sigma$ and define $\CG(\Sigma, \alpha)$ to be the full subgraph of $\CG(\Sigma)$ spanned by those curves that essentially intersect $\alpha$. However, in Theorem \ref{thm:curve} we will prove that, while $\CG(\Sigma, \alpha)$ has infinite diameter, it is never hyperbolic.


\section{Proof of Theorem \ref{thm:archyp}}

Throughout this section $\Sigma$ will be a connected orientable surface of infinite topological type, with empty boundary, and with at least one puncture.  Let $\Pi$ be the set of punctures of $\Sigma$, and regard its elements as marked points on $\Sigma$. We recall that, by the classification of infinite-type surfaces \cite{richards}, the set $\Pi$ is a subset of a Cantor set; see Proposition 5 of \cite{richards}. We will assume that $\Pi$ contains at least one element that is isolated in $\Pi$ equipped with the subspace topology. Let $P\subset \Pi$ be a nonempty finite set of isolated punctures; we are going to show that $\CA(\Sigma,P)$ is 7-hyperbolic.

\begin{remark} If $P$ contains a point that is not isolated, then $\CA(\Sigma,P)$ has finite diameter. 
Indeed, the fact that arcs are compact implies that, for any two arcs with an endpoint on the same $p \in P$ not isolated, one can find a third arc based at $p$ at distance at most 1 from both. 
\end{remark}
As we mentioned in the previous section, the main ingredient will be the result of Hensel-Przytycki-Webb \cite{HPW} stated as Proposition \ref{cor:hpw} above. In order to be able to make use of their result, we need to define a certain type of {\em  subsurface projection} map for arc graphs.\\


\noindent{\it Subsurface projections:}\\

Let $Y\subset \Sigma$ be a subsurface of finite topological type and complexity at least 2,  with $P\subset Y$. Let $\CA(Y,P)$ be the full subgraph of $\CA(\Sigma, P)$ spanned by those vertices of $\CA(\Sigma, P)$ that are entirely contained in $Y$. Fix, once and for all, an orientation on every boundary component of $Y$. We construct a map $\pi_{Y}$ from $\CA(\Sigma, P)$ to the power set of $\CA(Y, P)$ 
$$\pi_{Y}: \CA(\Sigma, P) \longrightarrow \mathcal{P}(\CA(Y, P))$$
as follows: 

\begin{itemize}[itemsep=2ex,leftmargin=0.5cm]
\item  If $c$ is entirely contained in $Y$, then $\pi_{Y} (c) : = \{c\}$. 

\item Otherwise, consider any subarc $c'$ of $c\subset Y$ that has one endpoint on $p\in P$. The other endpoint of $c'$  necessarily lies on a boundary curve $\gamma \subset \partial Y$ (if not,  $c'$ would be entirely contained in $Y$). We orient $c'$ so that it starts at $p$; recall that $\gamma$ has also been given an orientation. We define $\pi_{Y}(c')$ as the boundary of a regular neighborhood of  $c' \circ \gamma \circ {c'}^{-1}$, homotoped so that it is based at $p$. We set
 $$
\pi_{Y}(c) := \{\pi_{Y}(c') \mid  c' \text{ is a subarc of } c \text{ with one endpoint in } P\}.
$$
\end{itemize}

In particular observe that, for any $c\in \CA(\Sigma, P)$, the projection $\pi_{Y}(c)$ has at most two elements. 
The following lemma is an immediate consequence of the definition and will be key in the sequel.

\begin{lemma}
Let $Y\subset \Sigma$ be a subsurface of finite topological type, of complexity at least 2, and with $P\subset Y$. Consider the map $$\pi_{Y}: \CA(\Sigma, P) \to \mathcal P( \CA(Y, P))$$ just defined. Then:
\begin{enumerate}
\item For any vertex $c\in\CA(\Sigma, P)$, we have $${\rm diam}_{\CA(Y,P)} (\pi_{Y}(c))\le 2$$
\item Let $a,b \in \CA(\Sigma, P)$ be disjoint arcs. Then, for any $c_a \in \pi_{Y}(a)$ and any $c_b\in \pi_{Y}(b)$, we have $$d_{\CA(Y, P)}(c_a,c_b)\le 2$$
\end{enumerate} 
\label{l:project}
\end{lemma}

\begin{proof}
To prove $(i)$, let $c\in\CA(\Sigma, P)$ and consider $d_1\ne d_2\in \pi_{Y}(c)$, with $d_i$ coming from a subarc $c_i'$ of $c_i$, $i=1,2$. Observe that $c_1'$ and $c_2'$ have disjoint interiors and, since $d_1 \ne d_2$,  each $c_i'$ has  one endpoint in $P$ and the other one in a boundary component $\gamma_i \subset \partial Y$. If $\gamma_1 \ne \gamma_2$ then $d_{\CA(Y,P)}(d_1,d_2) =1$. Otherwise $d_1$ and $d_2$ intersect twice and, since $Y$ has complexity at least 2, there exists an arc in $Y$ that is disjoint from both $d_1$ and $d_2$. In particular, $d_{\CA(Y,P)}(d_1,d_2) =2$, as desired. 

Part $(ii)$ follows  from the proof of part $(i)$, since $a$ and $b$ are disjoint. 
\end{proof}

Since $\CA(Y,P) \subset \CA(\Sigma, P)$, we immediately obtain the following corollary of Lemma \ref{l:project}: 

\begin{corollary}
Let $Y\subset \Sigma$ be a subsurface of finite topological type, of complexity at least 2 and with $P\subset Y$. The inclusion map $$\CA(Y,P) \hookrightarrow \CA(\Sigma, P)$$ is a quasi-isometric embedding. More precisely, given $a,b \in \CA(Y,P)$
$$
d_{\CA(\Sigma, P)}(a,b) \le d_{\CA(Y,P)}(a,b) \leq 2 \, d_{\CA(\Sigma,P)}(a,b).
$$
\label{cor:qiembedding}
\end{corollary}

\begin{proof}
The left inequality comes from the fact that $\CA(Y,P) \subset \CA(\Sigma, P)$. To see that the right inequality holds, let $a,b \in \CA(Y,P)$ and consider a path $[a,b]$ from $a$ to $b$ in $\CA(\Sigma,P)$ between them. The projected path $\pi_Y([a,b])$ is a path from $\pi_Y(a) =a$ to $\pi_Y(b)=b$ in $\CA(Y,P)$ which, by Lemma \ref{l:project}, has length at most twice that of $[a,b]$. Hence the result follows.
\end{proof}

We need one more result before giving a proof of Theorem \ref{thm:archyp}:

\begin{lemma}
Let $F\subset \CA(\Sigma, P)$ be a finite set of arcs. Then there exists a connected finite-type subsurface $Y\subset \Sigma$, of  complexity at least 2 and with $P\subset Y$, such that every element of $F$ is entirely contained in $Y$. 
\label{lem:passtofinite}
\end{lemma}

\begin{proof}
First, we may enlarge $F$ into a bigger finite set $F'$ with the following two properties: 
\begin{enumerate}
\item Every pair of distinct elements of $P$ is the set of endpoints of an arc in $F'$. 
\item For any two pairs $(p_1,q_1)$ and $(p_2,q_2)$ of distinct elements of $P$, there exist arcs $c_1, c_2\in F'$ such that $c_i$ has endpoints in $(p_i,q_i)$, and $c_1$ and $c_2$ intersect.
\end{enumerate}
Since $\Pi$ is a subset of a Cantor set, $P$ is a finite set of isolated punctures and $F'$ is finite, there exists a regular neighborhood $Y$ of the union of the elements of $F'$ that is a finite-type surface. Moreover, up to a further (finite) enlargement of $F$, we may assume that $Y$ has complexity at least 2. We have $P \subset Y$ by (i); moreover, $Y$ is connected by (ii).  Since $F \subset F'$, it follows that every element of $F$ is entirely contained in $Y$. 
\end{proof}

We are finally ready to prove  Theorem \ref{thm:archyp}:

\begin{proof}[Proof of Theorem \ref{thm:archyp}]
We first prove that $\CA(\Sigma, P)$ is connected. Let $a, b \in \CA(\Sigma, P)$. Let $Y\subset \Sigma$ be the subsurface given by applying Lemma \ref{lem:passtofinite} to the set $F=\{a,b\}$, so that we may view $a,b$ as vertices of $\CA(Y,P)$. Since $\CA(Y,P)$ is connected, there exists a path from $a$ to $b$ in $\CA(Y,P)$ which, since  $\CA(Y,P)\subset \CA(\Sigma,P)$, gives the desired path between $a$ and $b$ in $\CA(\Sigma,P)$. 

To see that $\CA(\Sigma, P)$ has infinite diameter, choose a finite-type surface $Y\subset \Sigma$ of complexity at least 2 and with $P\subset Y$. Since $\CA(Y, P)$ has infinite diameter, and is quasi-isometrically embedded in $\CA(\Sigma, P)$ by Proposition \ref{cor:qiembedding}, it follows that the diameter of $\CA(\Sigma, P)$ is infinite. 

We finally prove that $\CA(\Sigma, P)$ is $\delta$-hyperbolic for a universal constant $\delta$; in fact, as mentioned in the introduction, in this particular case $\CA(\Sigma, P)$ will be 7-hyperbolic. Let $T \subset \CA(\Sigma,P)$ be a geodesic triangle, and let $F$ be the finite subset of $\CA(\Sigma,P)$ whose elements are the vertices of $T$. Let $Y$ be the connected, finite-type subsurface of $\Sigma$ yielded by applying Lemma \ref{lem:passtofinite} to the set $F$; in particular, we may view every vertex of $T$ as an element of $\CA(Y,P)$. Since $\CA(Y,P)$ is 7-hyperbolic, by Proposition \ref{cor:hpw}, it has a 7-center $c$, which is an arc in $\CA(Y,P)\subset \CA(\Sigma,P)$. Now the distances in $ \CA(\Sigma,P)$ from $c$ to the sides of $T$ are also at most $7$. Since $T$ is arbitrary, it follows that $\CA(\Sigma, P)$ is $7$-hyperbolic, as claimed.
\end{proof}

\section{Proof of Theorem \ref{thm:curve}}
\label{sec:curve}

In this section we will give a proof of Theorem \ref{thm:curve}. The strategy will be to prove that the graph $\CG(\Sigma, \alpha)$ contains a  quasi-isometric copy of a product of (three) hyperbolic spaces, namely arc graphs of subsurfaces determined by  $\alpha$. 

Let $\Sigma$ be a connected orientable surface, of complexity $\ge 1$ if it has finite topological type,  and fix a separating curve $\alpha \subset \Sigma$. Let $Y_\alpha$ be a closed regular neighborhood of $\alpha$ and denote by $Y_1, Y_2$ the connected components $\Sigma \setminus Y_\alpha$, so that $\Sigma= Y_1 \cup Y_2 \cup Y_\alpha$. Observe that $Y_i$ has one more puncture than $Y$; denote this new puncture by $p_i$.

Let $\CA(Y_i, p_i)$ be the simplicial graph whose vertices are those isotopy classes of  arcs in $Y_i$ with both endpoints in $p_i$, and edges correspond to pairs of such arcs with disjoint representatives. 

Similarly let $\CA(Y_\alpha)$ be the arc graph of $Y_\alpha$, which may be defined as follows. Fix, once and for all, a point on each boundary component of $Y_\alpha$, which we denote $m^+$ and $m^-$, respectively. The vertices of  $\CA(Y_\alpha)$ are isotopy classes, relative endpoints, of arcs with one endpoint on $m^+$ and the other on $m^-$, and two such arcs span an edge if they have disjoint interiors. We note that $\CA(Y_\alpha)$ is isomorphic (and thus isometric) to $\mathbb Z$ (as its infinite, connected and degree 2 in every vertex). Theorem \ref{thm:curve} will follow once we have proved the following: 

\begin{proposition} The graph $\CG(\Sigma, \alpha)$ contains a quasi-isometrically embedded copy of $\CA(Y_1, p_1)\times \CA(Y_2, p_2) \times \CA(Y_\alpha).$
\label{prop:product}
\end{proposition}

\begin{proof}
We are going to construct an explicit quasi-isometric embedding \[ \psi: \CA(Y_1, p_1)\times \CA(Y_2, p_2) \times \CA(Y_\alpha)\to \CG(\Sigma, \alpha).\] 
In order to do so it will be convenient to have an   alternate description of the graphs $\CA(Y_i,p_i)$, which we now give. 

Let $\overline{Y}_i$ be the two connected components of $\Sigma \setminus {\rm int}(Y_\alpha)$, noting that $\overline Y_i = {\rm int}(Y_i)$ for $i=1,2$. Let $\alpha_i$ be the boundary curve of $\overline Y_i$ that is isotopic to $\alpha$ in $\Sigma$. Consider the simplicial graph $\CA(\overline Y_i,\alpha_i)$ whose vertices are isotopy classes of arcs on $\overline Y_i$ with both endpoints on $\alpha_i$, where isotopies need not fix $\alpha_i$ pointwise, and where two such arcs are adjacent in  $\CA(\overline Y_i,\alpha_i)$ if they can be realized disjointly. Observe that  the graphs  $\CA(\overline Y_i,\alpha_i)$  and $\CA(Y_i, p_i)$ are naturally isomorphic. 

Armed with this alternate description, we proceed to construct the desired map $\psi$. Let $(a_1,a_2,b)\in \CA(Y_1, p_1)\times \CA(Y_2, p_2) \times \CA(Y_\alpha)$, where we view the arc $a_i$ as an element of $\CA(\overline Y_i, \alpha_i)$ instead. We choose an arc $b'\in\CA(Y_\alpha)$ that is disjoint from (but possibly equal to) $b$, and we glue the arcs $a_1, b,b'$ and $a_2$ into a simple closed curve as shown in Figure \ref{fig:glue}. In this way we obtain 
the desired curve $\psi(a_1,a_2,b) \in \CG(\Sigma, \alpha)$. We stress that while such curve is not unique, any two of them are within a uniformly bounded distance since they intersect a uniformly bounded number of times.  For a totally analogous reason we deduce that there exists $L\ge 1$ such that $\psi$ is $L$-Lipschitz, that is: 
$$
d(\psi (a_1,a_2,b), \psi (a_1',a_2',b')) \le L \cdot d((a_1,a_2,b), (a_1',a_2',b'))
$$
for all triples $(a_1,a_2,b)$ and $(a_1',a_2',b')$. 

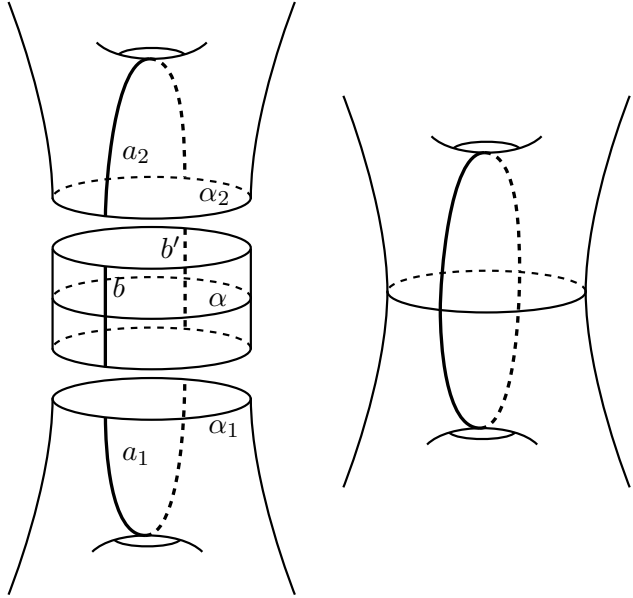
\begin{figure}[h]
\leavevmode \SetLabels
\L(.28*.50) $b$\\
\L(.33*.565) $b'$\\
\L(.29*.73) $a_2$\\
\L(.29*.23) $a_1$\\
\L(.37*.66) $\alpha_2$\\
\L(.38*.275) $\alpha_1$\\
\L(.38*.485) $\alpha$\\
\endSetLabels
\begin{center}
\AffixLabels{\centerline{\epsfig{file =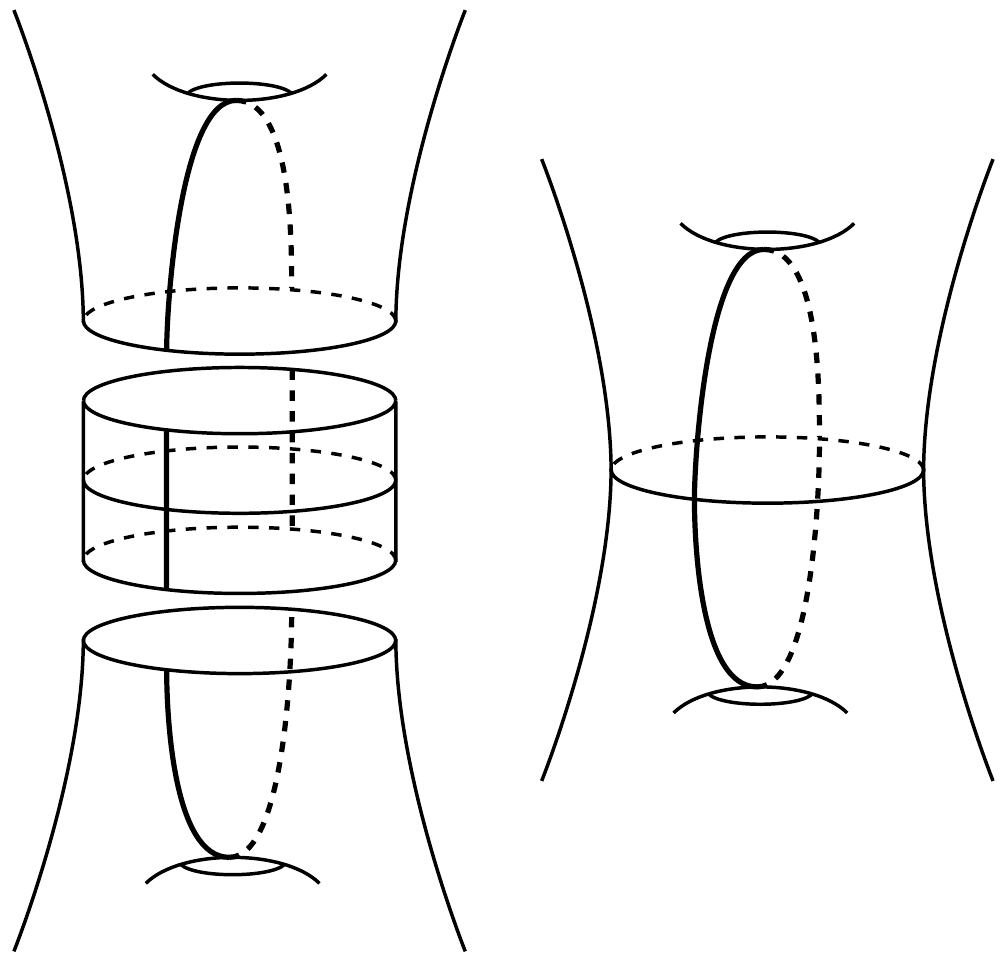, height=8cm, angle=0} }}
\vspace{-10pt}
\end{center}
\caption{Gluing arcs into a simple closed curve.} \label{fig:glue}
\end{figure}

In order to finish the proof that $\psi$ is a quasi-isometric embedding, we first construct a map 
 \[\pi= \CG(\Sigma, \alpha) \to \CA(Y_1, p_1)\times \CA(Y_2, p_2) \times \CA(Y_\alpha)\] as follows. 
 For $i=1,2$, we define a map
\[\pi_i: \CG(\Sigma, \alpha) \to \mathcal P( \CA(Y_i, p_i)) \]
given by $\pi_i(\beta) = \beta \cap Y_i$. Similarly, we define a map
\[\pi_\alpha: \CG(\Sigma, \alpha) \to  \mathcal P(\CA(Y_\alpha)) \] 
into the power set of $\CA(Y_\alpha)$ as follows. Fix, once and for all, an orientation on each of the boundary components of $Y_\alpha$, which we denote $\alpha^+$ and $\alpha^-$ respectively. Now, for any curve $\beta \in \CG(\Sigma, \alpha)$ and any connected component $b$ of $\beta \cap Y_\alpha$, we define $\pi_\alpha(b)$ to be the arc of $\CA(Y_\alpha)$ that starts from $m^+$, follows $\alpha^+$ until meeting $b$, then follows $b$ until meeting $\alpha^-$, and finally follows $\alpha^-$ until meeting $m^-$; recall $m^{\pm}$ are the points used to define $\CA(Y_\alpha)$. Armed with this definition, we set 
$$\pi_\alpha(\beta) = \{\pi_\alpha(b) \mid b \text{ is a connected component of } \beta \cap Y_\alpha \};$$ note that $\pi_\alpha(\beta)$ has diameter at most 1. Finally, we set \[\pi= (\pi_1,\pi_2, \pi_\alpha): \CG(\Sigma, \alpha) \to \CA(Y_1, p_1)\times \CA(Y_2, p_2) \times \CA(Y_\alpha). \]
(Roughly speaking, the map $\pi$ can be seen as a natural way of decomposing every curve $\beta$ essentially intersecting $\alpha$  
by a certain power of the Dehn twist along $\alpha$ and its behaviour away from $\alpha$.)
Since disjoint curves on $\CG(\Sigma, \alpha)$ project to disjoint arcs on $Y_i$ and $Y_\alpha$, respectively, we deduce that $\pi_i(\beta)$ and $\pi_\alpha(\beta)$ are sets of diameter 1 in $\CA(Y_i,p_i)$ and $\CA(Y_\alpha)$, respectively. For the same reason, using an argument analogous to that of Proposition \ref{cor:qiembedding}, we deduce that the maps $\pi_i$ and $\pi_\alpha$ are 1-Lipschitz, and therefore we have: 

\medskip

\noindent{\bf {Fact.}} The map $\pi$ is 3-Lipschitz: for every $\beta,\gamma \in \CG(\Sigma, \alpha)$ we have \begin{eqnarray}d(\pi(\beta),\pi(\gamma)) \le 3\cdot d(\beta,\gamma).\end{eqnarray}
Moreover, it is immediate from the constructions of the maps $\pi$ and $\psi$ that there exists a universal constant $R>0$, not depending on $\Sigma$ or $\alpha$, such that 
\begin{eqnarray}
d((\pi\circ \psi) (a_1,a_2, b), (a_1,a_2,b)) \le R
\end{eqnarray}
for all triples $(a_1,a_2,b)$.   

We are finally in a position to prove that  $\psi$ is a quasi-isometric embedding. As $\psi$ is Lipschitz it suffices to prove that there exist $\lambda \ge 1$ and $C\ge 0$ such that 
\begin{eqnarray*}\frac{1}{\lambda}\cdot d((a,b,c), (a',b',c')) - C \le d(\psi((a,b,c), \psi(a',b',c')),
\label{eq:lower}
\end{eqnarray*} 
for all triples $(a,b,c)$ and $(a',b',c')$ in $\CA(Y_1, p_1)\times \CA(Y_2, p_2) \times \CA(Y_\alpha)$. Using (1) and (2) above, we obtain: 
\begin{eqnarray*}
d((a,b,c), (a',b',c')) &\le& d((\pi \circ \psi) (a,b,c), (\pi \circ \psi) (a',b',c')) + 2R \\ &\le& 3 d(\psi (a,b,c), \psi (a',b',c')) + 2R \\ &\le&  3 d(\psi (a,b,c), \psi (a',b',c')) + 6R
\end{eqnarray*}
and one can take $\lambda =3 $ and $C=2R$. This finishes the proof of Proposition \ref{prop:product}.
\end{proof}

\end{document}